\documentclass{amsart}
\usepackage[foot]{amsaddr}
\usepackage[T1]{fontenc}
\usepackage{textcomp}
\usepackage[utf8x]{inputenc}
\usepackage[english]{babel}
\usepackage{ucs}
\usepackage{graphicx}
\usepackage{mathrsfs}
\usepackage{rotating} 
\usepackage{tikz}
	\usetikzlibrary{cd}
\usepackage{bm}
\usepackage{amssymb}
\usepackage[titletoc,toc]{appendix}
\usepackage{subcaption}
\numberwithin{equation}{section} 

\usepackage{tabu}
\usepackage{tcolorbox}
\usepackage{xcolor}
\definecolor{linkred}{rgb}{0.6,0,0}
\usepackage[
	pdftitle={Explicit Moduli of Superelliptic Curves with Level Structure},
	pdfauthor={Olof Bergvall and Oliver Leigh},
	ocgcolorlinks,
	linkcolor=black,
	citecolor=linkred,
	urlcolor=linkred]
{hyperref}

\usepackage{marginnote}
\newtheorem{thm}{Theorem}[section]

\newtheorem{cor}[thm]{Corollary}
\newtheorem{lem}[thm]{Lemma}
\newtheorem{prop}[thm]{Proposition}

\theoremstyle{definition}
\newtheorem{definition}[thm]{Definition}
\newtheorem{rem}[thm]{Remark}
\newtheorem{re}[thm]{}
\newtheorem{tableRem}[thm]{Table}
\newtheorem{convention}[thm]{Convention}
\numberwithin{equation}{section}

\newcommand{\F}{\mathbb{F}}

\newcommand{\C}{\mathbb{C}}

\newcommand{\Sp}{\mathrm{Sp}}

\newcommand{\PP}{\mathbb{P}}
\newcommand{\FF}{\mathbb{F}}

\newcommand{\ZZ}{\mathbb{Z}}

\newcommand{\CC}{\mathbb{C}}

\newcommand{\jac}{\mathrm{Jac}}

\newcommand{\Hyp}{\mathrm{Hyp}}
\newcommand{\Sup}{\mathrm{Sup}}

\newcommand{\Jac}{\mathrm{Jac}}

\makeindex
\begin{document}

\title[Superelliptic Curves with Level Structure]{Explicit Moduli of Superelliptic Curves with Level Structure}%
\author{Olof Bergvall}%
\author{Oliver Leigh}
\email{olof.bergvall@hig.se}
\email{oliver.leigh@math.su.se}
\address{Dept. of Electronics, Mathematics and Natural Sciences, University of G\"avle}
\address{Dept. of Mathematics, Stockholm University}

\begin{abstract}
In this article we give an explicit construction of the moduli space of trigonal superelliptic curves with level 3 structure. 
The construction is given in terms of point sets on the projective line and leads to a closed formula for the number of connected (and irreducible) components of the moduli space.
The results of the article generalise the description of the moduli space of hyperelliptic curves with level 2 structure,  due to Dolgachev and Ortland, Runge and Tsuyumine.
\end{abstract}

\maketitle

\section{Introduction}
As Mumford describes in \cite[\S2]{mumford_tata2}, $2$-torsion divisors on a hyperelliptic curve correspond precisely to degree zero linear combinations of ramification points. 
Hence, if one takes distinct points $P_1, \ldots, P_{2g+2}$ on $\PP^1$ and  considers $C$, the unique hyperelliptic curve ramified over these points, then one can hope to explicitly describe symplectic bases for the $2$-torsion Jacobian $\Jac(c)[2]$ in terms of $P_1, \ldots, P_{2g+2}$.
Indeed, after choosing an ordering for the branch points, it turns out that there is a natural way to obtain a (full) symplectic level $2$ structure on $C$ from combinations of $P_1, \ldots, P_{2g+2}$.
One can then obtain any level 2 structure on $C$ via the symplectic group $\Sp(2g,\FF_2)$.
On top of this, since each choice of ordering will give rise to a different symplectic level $2$, this construction naturally defines an embedding of the symmetric group $S_{2g+2}$ into the symplectic group $\Sp(2g,\FF_2)$. 

Using this construction, Dolgachev and Ortland \cite{dolgachevortland} considered $\Hyp_g[2]$, the moduli space of hyperelliptic curves with level $2$ structure. They showed that each irreducible component of  $\Hyp_g[2]$ is isomorphic to $M_{0,2g+2}$, the moduli space of smooth rational curves with $2g+2$ distinct markings, and that the irreducible components  of  $\Hyp_g[2]$ are indexed by the set of cosets $\mathscr{C}:=\Sp(2g,\FF_2)/S_{2g+2}$.
Dolgachev and Ortland also posed the question: \textit{What are the connect components of $\Hyp_g[2]$?}
This question was answered independently by Tsuyumine \cite{tsuyumine} and Runge \cite{runge}, who showed that the irreducible components are also the connected components. 
Thus, if we let $X_c$ be a variety isomorphic to $M_{0,2g+2}$ and indexed by the coset $c \in \mathscr{C}$, then there is an isomorphism
\begin{equation*}
 \Hyp_3[2] \cong \bigsqcup_{c \in \mathscr{C}} X_c.
\end{equation*}

It is natural to ask which parts of the above discussion extend to other types of ramified covers.
In the present article we proceed in this direction by investigating cyclic covers of $\PP^1$ called \textit{superelliptic curves}.
We pay special attention to the simplest non-hyperelliptic case, namely trigonal curves. 
We find (somewhat surprisingly) an almost perfect analogy between the case of cyclic trigonal curves with symplectic level $3$ structure and hyperelliptic curves with level $2$ structure.
Our main result is the following.

\begin{thm}
\label{mainthm}
Let $g$ be a positive integer and take the unique integers $n$ and $k$ such that $0 \leq k \leq 2$ and $g+2=3n-k$.
The moduli space $\Sup^{3}_{g}[3]$ of trigonal superelliptic curves with level $3$ structure has 
\[
|\Sp(2g,\FF_3)| \left(
\sum_{0\leq i < \left\lceil \frac{g+2-2k}{6}  \right\rceil} \frac{1}{ (3i +k )!  (g+2 -( 3i + k) )! } 
~~+~~
\sum_{\left\{i\in  \ZZ \,:\, i= \frac{g}{2}+1 \right\}  } \frac{1}{ 2 (i!)^2} 
\right)
\]
 connected components. Each connected component is irreducible and is isomorphic to the moduli space $M_{0,g+2}$
 of smooth rational curves with $g+2$ marked points. 
\end{thm}

It is plausible that the above theorem can be used to pave a way towards the conjectures of Bergstr\"om-van der Geer \cite[Sec. 12]{bergstromvandergeer}
and the present work may be useful in the moduli theory related to $2$-spin Hurwitz numbers via \cite{leigh_div_ram} and \cite{leigh_loc}.
Other possible directions for future work include cohomological computations, see e.g. \cite{bergstrombergvall} and \cite{bergvall_gd},
and investigations of special kinds of superelliptic curves, e.g. Picard curves, Belyi curves or plane curves.
Relevant previous work include Accola's characterization of cyclic trigonal curves  in terms of vanishing properties of theta functions \cite{accola},
Kontogeorgis' study of automorphism groups of rational function fields \cite{kontogeorgis},
Kopeliovich's computations of Thomae formulae for cyclic covers of the projective line \cite{kopeliovich}
and Wangyu's characterization of cyclic covers of the projective line with prime gonality \cite{wangyu}.
For further references, and a survey of the field, see \cite{malmendiershaska}.

The paper is structured as follows. In Sections~\ref{background_hyperelliptic} and \ref{background_sup_sec} we
review the necessary background on hyperelliptic and superelliptic curves.
We also rephrase some of the classical material into a setup which is mutually compatible
and generalizable. In Section~\ref{divisors_section} we study divisors on superelliptic curves.
In particular, we study divisors generated by ramification points and how the Weil pairing
behaves on pairs of such divisors. Finally, in Section~\ref{trig_sec} we specialize to
the case of trigonal curves. Among other results, we obtain an explicit description of the moduli space
of superelliptic trigonal curves with level $3$ structure and an explicit
formula for the number of connected (and irreducible) components of this space (see Theorem~\ref{mainthm} above).

\section{Background on Hyperelliptic Curves}\label{background_hyperelliptic}

\begin{re}[\textit{Overview of section}]
In this section we will recall facts about hyperelliptic curves and give an overview of an explicit construction of the moduli space of hyperelliptic curves with level 2 structure. This construction, which is based upon finite  subsets of the projective line, can be mainly attributed to Mumford \cite{mumford_tata2} and Dolgachev-Ortland \cite{dolgachevortland}. We will present this material in a way that can be readily generalised to superelliptic curves. 
\end{re}

\begin{definition}[\textit{Hyperelliptic curves}] \label{def_hyperelliptic}
A \textit{hyperelliptic curve} is a degree $2$ morphism $\pi: C \to \PP^1$ such that $C$ is a smooth (connected) curve of genus $g>1$. 

Two hyperelliptic curves $\pi:C\rightarrow \PP^1$ and $\pi':C'\rightarrow \PP^1$ are isomorphic if there are isomorphisms $\alpha:C\rightarrow C'$ and $\beta: \PP^1 \rightarrow \PP^1$ such that $\pi \circ \alpha = \beta \circ \pi'$. 
\end{definition}

\begin{rem}[\textit{Equivalent defintion of hyperelliptic curve}] \label{rem_hyperelliptic_definiton}
 A much more standard definition of a \textit{hyperelliptic curve} is a curve $C$ with genus $g>1$ such that a degree $2$ morphism $\pi: C \to \PP^1$ \textit{exists}. However, it is well known that any such degree $2$ morphism is unique up to isomorphism, so the two definitions are equivalent (see for example \cite[IV Prop 5.3]{hartshorne}). We use the definition of hyperelliptic curves from definition \ref{def_hyperelliptic} since it is more in-line with the definition of \textit{superelliptic curves} which will be given in definition \ref{superelliptic_def}.  
\end{rem}

\begin{re}[\textit{Moduli space of hyperelliptic curves}] 
Definition \ref{def_hyperelliptic} extends naturally to give a moduli functor. We denote by $\Hyp_g$  the (coarse) moduli space parametrising  hyperelliptic curves. Moreover, by the uniqueness of the hyperelliptic morphism pointed out in remark \ref{rem_hyperelliptic_definiton}, we have that there exists a natural inclusion $\Hyp_g \hookrightarrow M_g$. The image of this inclusion is called the \textit{hyperelliptic locus}. 
\end{re}

\begin{rem}\label{remark_hyp_stack} 
In this article the main focus is on the coarse moduli \textit{scheme} of hyperelliptic curves $\Hyp_g$  instead of the moduli \textit{stack} of hyperelliptic curves $\mathcal{H}\hspace{-0.1em}\mathit{yp}_g$. A key difference between the two is that $\mathcal{H}\hspace{-0.1em}\mathit{yp}_g$ keeps track of the automorphisms of the hyperelliptic curves.  The stack theoretic viewpoint for this space and the analogous superelliptic case (appearing in section \ref{background_sup_sec}) was considered in \cite{arsievistoli}. 
\end{rem}

\begin{rem}
Since we have defined $\Hyp_g$ as a moduli space of maps, this moduli space can also naturally be viewed as the \textit{Hurwitz space} $\mathrm{Hur}_g(2)$. 
\end{rem}

\begin{re}[\textit{Affine models for hyperelliptic curves}] \label{hyp_affine_model}
It is a well known result that a genus $g$ hyperelliptic curve $\pi: C \to \PP^1$ has a non-unique affine model of the form
\[
s^2 - f(t)=0 
\]
where $f \in \CC[t]$ has unique roots and is either of degree $2g+1$ or $2g+2$.
(see for example \cite[p. 3.28]{mumford_tata2}). 
Moreover, if $C^\circ$ is the smooth locus of the above affine variety, then $C$ is the the smooth completion of $C^\circ$ and $\pi$ is the associated projection map arising from the projection to the $t$ coordinate. 

A second chart for $C$ can be defined by the equation $0 = u^p - h(v)$
where $h \in\CC[v]$ is the unique polynomial with $ h(v) = v^{2g+2} f(\frac{1}{v})$. The change of coordinates is then defined by $(t,s)  \mapsto \left(\frac{1}{v} ,\frac{u}{v^{g+1}}\right)$.

The equivalence relation between two isomorphic affine models is described via 
\[
s^2 - f(t) ~~ \sim ~~ s^2 - (ct+d)^{2g+2} f\left(\frac{at+b}{ct+d}\right) 
\hspace{0.5cm}
\mbox{for}
\hspace{0.5cm}
\left(\begin{array}{cc} a & b \\ c & d \end{array} \right) \in \mathrm{GL}(2,\C). 
\]
Note that if $s^2 - h(t)$ is the resulting polynomial on the right hand side, $f$ and $h$ can have degrees differing by $1$ depending on whether a root was moved to or away from infinity. 
\end{re}

\begin{re}[\textit{Construction via choosing ramification points}] \label{construct_hyperelliptic_from_points}
Consider the relationship between the affine model and a hyperelliptic curve $\pi:C\rightarrow \PP^1$ described in \ref{hyp_affine_model}. An immediate observation from this relationship is that the ramification points of $\pi$ occur at the roots of $f(t)$ for $\deg(f) = 2g+2$ and also at $\pi^{-1}(\infty)$ if $\deg(f) = 2g+1$.

A result of this observation is that one can construct a unique hyperelliptic curve by choosing a set of $2g+2$ points in $\PP^1$. Taking into account the equivalence described in \ref{hyp_affine_model} we arrive at the following theorem.
\end{re}

\begin{thm}[\textit{Scheme structure of $\Hyp_g$,  \cite{fischer, lonstedkleiman}}] \label{isomorphism_hyperelliptic_moduli_scheme}
There is an isomorphism of schemes 
\[
M_{0,2g+2}/S_{2g+2} \longrightarrow \Hyp_g.
\] 
\end{thm}

\begin{rem}
Versions of the isomorphism from theorem \ref{isomorphism_hyperelliptic_moduli_scheme}  also exist for compact moduli moduli spaces  e.g. \cite[Ex. 6.25]{harrismorrision} and \cite[Cor. 2.5]{avritzerlange}.
\end{rem}

\begin{convention}[\textit{Notation for choice of hyperelliptic curve}]\label{hyperelliptic_convention}
For the remainder of this section we will assume that $\pi:C\rightarrow \PP^1$ is a hyperelliptic curve of genus $g$. Moreover,  given the equivalence of affine models described in \ref{hyp_affine_model},  we will make the assumption that $\infty \in \PP^1$ is a branch point of $\pi$ and the affine model is 
\[
0= s^2 - \prod_{i=1}^{2g+1} (t- a_i).
\]
We will denote the ramification points as $Q_i := \pi^{-1}(a_i)$ and  $Q_\infty:= \pi^{-1}(\infty)$. 
\end{convention}

\begin{re} [\textit{Natural principal divisors}]\label{natural_principal_hyperelliptic_divisors}
Recalling convention \ref{hyperelliptic_convention} we observe that there are natural principal divisors on $C$ given by:
\begin{enumerate}
\item \textit{Horizontal:} $(s) = \left(\tfrac{u}{v^{g+1}}\right) = \sum_{i=1}^{2g+1}(Q_i  -  Q_\infty). $
\item \textit{Vertical ramified:} $(t-a_i) = \left(\tfrac{1-a_i u}{u}\right) = 2 \,Q_i  -  2\, Q_\infty. $
\item \textit{Vertical unramified:} $(t-a) = \left(\tfrac{1-a u}{u}\right) = \pi^{*}P_a  -  2\, Q_\infty $.
\end{enumerate}
Here $a\in \CC^*\setminus\{a_1,\ldots,a_{2g+1}\}$ and $P_a\in\PP^1$ is the corresponding divisor. The divisor $\pi^{*}P_a$ consists of $2$ distinct points.

\end{re}

\begin{re}[\textit{$2$-torsion in the Jacobian of a hyperelliptic curve}]\label{re_hyperelliptic_Delta_defintion}
A key focus of this section is related to the study of $2$-torsion within the Jacobian of hyperelliptic curves. In this direction, an immediate observation from \ref{natural_principal_hyperelliptic_divisors} is that for $i\in\{1,\ldots,2g+1\}$ we have
\[
2 \cdot (Q_i  -  Q_\infty) \sim 0.
\]
This shows that each $D_i:=[Q_i - Q_\infty]$ is a natural  element of $\Jac(C)[2]$.  We can also now consider the natural $\FF_2$-vector subspace spanned by the classes
\[
\Delta := \FF_2\mbox{-}\mathrm{Span}\Big\{ D_1, \ldots, D_{2g+1} \Big\} \subseteq \Jac(C)[2].
\]
On top of this, the horizontal principal divisor from \ref{natural_principal_hyperelliptic_divisors} gives the  relationship
$
0 = \sum_{i=1}^{2g+1} D_i. 
$
This shows that $\Delta$ is spanned by any choice of $2g$ classes from $ D_1, \ldots, D_{2g+1}$. Indeed, it was shown by Mumford that any of these choices 
gives a basis for $\Delta$ since $\Delta\cong \FF_2^{2g}$. Moreover, combining this result with a basis constructed by Dolgachev and Ortland gives the following key result. 
\end{re}

\begin{prop}[\textit{A symplectic basis for $\Jac(C)[2]$, \cite[IIIa Lem. 2.5]{mumford_tata2} \& \cite[\S3 Lem. 2]{dolgachevortland}}] \label{prop_hyperelliptic_symplectic_basis_Jac2}

There are $\FF_2$-linear isomorphisms  $\Jac(C)[2] \cong \Delta \cong \FF_2^{2g}$. The basis $(A_1,\ldots,A_g, B_1,\ldots,B_g)$ of $\Jac(C)[2] $ defined by
\[
A_i := D_{2i-1} + D_{2i}  
\hspace{0.75cm}
\mbox{and}
\hspace{0.75cm}
B_i := D_{2i} +\cdots+ D_{2g+1},
\]
is symplectic with respect to the Weil Pairing. Moreover, the linear map taking this basis to the standard symplectic basis for $\FF_2^{2g}$ is an isometry.
\end{prop}

\begin{rem}
The methods used in \cite[\S3 Lem. 2]{dolgachevortland} to prove Proposition \ref{prop_hyperelliptic_symplectic_basis_Jac2}  are more analytic in nature than the methods employed in this article. Indeed, the authors of \cite{dolgachevortland} consider $\pi:C\rightarrow \PP^1$ as a two-sheeted covering and consider a symplectic basis of cycles in $H_1(C,\ZZ)$ corresponding to the basis elements of Proposition \ref{prop_hyperelliptic_symplectic_basis_Jac2}. 

To be more precise, $A_i$ corresponds to a path which goes from $Q_{2i-1}$ along one sheet to $Q_{2i}$ and then returns along the other sheet, while $B_i$ corresponds to a path which goes from $Q_{2i}$ along one sheet to $Q_{2g+1}$ and then returns along the other sheet. The symplectic basis of cycles in $H_1(C,\ZZ)$ is then combined with a normalised basis of $H^0(C,\Omega_C)$ to construct a \textit{branch point period matrix}. 
\end{rem}

\begin{re}[\textit{Inclusion of the symmetric group into the symplectic  group}]\label{hyperelliptic_symmetric_group_inclusion}
The symmetric group $S_{2g+2}$ permutes the points $Q_1,\ldots, Q_{2g+1},Q_{\infty}$ and thus defines an action on the set $\mathsf{D}:=\{D_1,\ldots,D_{2g}\}$. Since any permutation of $Q_1,\ldots, Q_{2g+1},Q_{\infty}$ will give the same hyperelliptic curve, the action of $S_{2g+2}$ on $\mathsf{D}$ extends to give a change of symplectic basis. This defines a group homomorphism $S_{2g+2} \rightarrow \Sp(2g,\FF_2)$ which is in fact an embedding (see \cite[p. 60]{bergvallthesis} for details and further references).
\end{re}

 \begin{re}[\textit{Natural morphism to $\Hyp_{g}[2]$}]
 We denote the (coarse) \textit{moduli space of hyperelliptic curves with level-2 structure} by $\Hyp_{g}[2]$. This is a moduli space parametrising pairs consisting of a hyperelliptic curve $\pi:C\rightarrow \PP^1$ and an isometry $\eta: \FF_2^{2g} \rightarrow \Jac(C)[2]$, where $\FF_2^{2g}$ uses the standard symplectic form and $\Jac(C)[2]$ uses the Weil pairing. Now, Proposition \ref{prop_hyperelliptic_symplectic_basis_Jac2}  (which extends easily to families) naturally gives a morphism
 \[
 M_{0,2g+2} \rightarrow \Hyp_g[2].
 \]
Moreover, it is shown in \cite[\S3 Thm. 1]{dolgachevortland} that this morphism is an isomorphism onto an irreducible  component of $\Hyp_g[2]$. On top of this, if we 
consider the morphism that forgets the isometry, $\Hyp_g[2] \rightarrow \Hyp_g$, then we obtain the following commutative diagram. 
\[
\begin{tikzcd}
M_{0,2g+2} \arrow[r, hook] \arrow[d]
& \Hyp_{g}[2] \arrow[d] \\
M_{0,2g+2}/S_{2g+2} \arrow[r, "\cong"]
& \Hyp_{g}
\end{tikzcd} 
\]
In fact, the morphism $ M_{0,2g+2} \rightarrow \Hyp_g[2]$ actually defines an isomorphism to a connected component of $\Hyp_g[2]$. Using this construction one can describe all of the of connected components of $\Hyp_g[2]$  and arrive at the following theorem. 
\end{re}

\begin{thm}[\textit{Decomposition of $\Hyp_g[2]$, \cite[Thm. 2]{tsuyumine} \& \cite[Thm. 4.1]{runge}}] \label{Hyperelliptic_w_level_structure_decomposition_theorem}  
Consider  the inclusion of the symmetric group in the symplectic group described in \ref{hyperelliptic_symmetric_group_inclusion} and denote the quotient set $\mathrm{Sp}(2g,\FF_2)/S_{2g+2}$ by
$\mathscr{C}$. Then, if $X_{c}$ denotes a copy of $M_{0,2g+2}$ indexed by $c \in \mathscr{C}$,  there is an isomorphism of schemes
\begin{equation*}
\mathrm{Hyp}_g[2] \overset{\cong}{\longrightarrow} \bigsqcup_{c \in \mathscr{C}} X_c,
\end{equation*}
so $\mathrm{Hyp}_g[2]$ is smooth and the number of connected components of $\mathrm{Hyp}_g[2]$ is given by $|\mathrm{Sp}(2g,\FF_2)|/|S_{2g+2}|$.
\end{thm}

\begin{rem}
The result of Tsuyumine from \cite[Thm. 2]{tsuyumine} is actually slightly different from the one stated in theorem \ref{Hyperelliptic_w_level_structure_decomposition_theorem}. He shows the more general result that $\mathrm{Hyp}_g[n]$ has $|\mathrm{Sp}(2g,\FF_2)|/|S_{2g+2}|$ connected components whenever $n$ is divisible by $2$. However, components are only explicitly computed in the case $n=2$.
\end{rem}

\begin{rem}
Modulo minor mistakes (pointed out in \cite{runge}), Dolgachev-Ortland \cite[\S3]{dolgachevortland} presented a weaker version of the above theorem showing that the \textit{irreducible components} of $\mathrm{Hyp}_g[2]$ are indexed by $\mathscr{C}$  and are isomorphic to $M_{0,2g+2}$.
\end{rem}

\begin{rem}
Historically, a key purpose of level structures is to rigidify moduli problems. That is, to eliminate the automorphism groups of the objects being parametrised. However, one observation from the constructions in this section is that level-2 structures do not rigidify hyperelliptic curves. To see this, we note that the symplectic basis given in Proposition \ref{prop_hyperelliptic_symplectic_basis_Jac2} is invariant under the hyperelliptic involution. 
\end{rem}

\section{Background on Superelliptic Curves}\label{background_sup_sec}

\begin{re}
In this section we will take the results from section \ref{background_hyperelliptic} as inspiration and consider a generalisation of hyperelliptic curves.
To be precise, we will consider cyclic morphisms of degree $p$ where $p$ is a prime.  We note, however, that many of these concepts can also be extended to the non-prime case. 
We will also assume that $g>0$. 
\end{re}

\begin{definition}[\textit{Superelliptic curves}]\label{superelliptic_def}
A \textit{superelliptic curve} is a  degree $p$ morphism $\pi: C \to \PP^1$ such that $C$ is a smooth (connected) curve and the Galois group of $\pi$ is cyclic. Two superelliptic curves $\pi_1: C_1 \to \PP^1$ and $\pi_2: C_2 \to \PP^1$ are said to be \textit{equivalent} if there are isomorphisms $\psi: C_1 \rightarrow C_2$ and $\varphi:\PP^1 \rightarrow \PP^1$ such that $\pi_2 \circ \psi = \varphi \circ \pi_1$. 
\end{definition}

\begin{re}[\textit{Moduli space of superelliptic curves}]  
Definition \ref{superelliptic_def} extends naturally to give a moduli functor. We denote by $\Sup^p_g$  the (coarse) moduli space of superelliptic curves
of degree $p$ and genus $g$. Here are some well known relationships with other moduli spaces:
\begin{enumerate}
\item For $d=2$ and $g>1$, definition \ref{superelliptic_def} matches with that of hyperelliptic curves and gives rise to an isomorphism $\Hyp_g \cong \Sup_g^2$. This is because a hyperelliptic curve  uniquely determines a  morphism $\pi:C\rightarrow \PP^1$ of degree 2 and because the Galois group of such a morphism is automatically cyclic. 
\item There is a natural morphism $\Sup^p_g \rightarrow M_g$ to the (coarse) moduli space of smooth genus $g$ curves. In general, this morphism is not an immersion. However, it is an immersion in special cases such as for $d=3$ when $g>4$. 
\item For $g>2$, a curve cannot be both hyperelliptic and degree 3 superelliptic, meaning that the images of $\Sup_g^2$ and $\Sup_g^3$ are disjoint in $M_g$. 
\end{enumerate}
\end{re}

\begin{re}[\textit{Ramification of superelliptic curves}]\label{re_superellitpic_ramification}
Degree $p$ superelliptic curves have the property of being totally ramified at each ramification point. In other words, each branch point has a unique preimage. Hence, by the Riemann-Hurwitz formula the number of ramification points is 
\begin{equation*}
m = \frac{2g}{p-1}+2.
\end{equation*}
This places a strong condition on the possible choices for $g$ and $p$. An equivalent condition on $g$ is obtained by specifying $m$ and $p$ to give
\[
 g = \frac{(p-1)(m-2)}{2}.
 \]
We note that when $p=2$ or $p=3$, superelliptic curves exist for all $g$. Indeed, for $d=2$ we recover $m=2g+2$ as in Section \ref{background_hyperelliptic} and when $d=3$ we obtain $m=g+2$.
\end{re}

\begin{re}[\textit{Affine model for superelliptic curves}] \label{sup_affine_model}
It is a well know result (see for example \cite[Sec. 5]{malmendiershaska})
that a superelliptic curve $\pi: C \to \PP^1$ has a non-unique affine model given by 
\[
0= s^p - f(t) 
\]
where $f \in \CC[t]$ has roots with orders less than $p$. In particular, if $C^\circ$ is the smooth locus of the above affine variety, then $C$ is the the smooth completion of $C^\circ$ and $\pi$ is the associated projection map arising from the projection to the $t$ coordinate. 

We can construct an affine cover of $C$ by considering a second chart. First, define an integer via the ceiling function $n=:\lceil \deg(f)/p \rceil$ and also $\kappa := pn - \deg(f)$. We now define the second chart by
\[
0 = u^p - v^{\kappa} h(v)
\]
where $h \in\CC[v]$ is the unique polynomial with $ h(v) = v^{\deg(f)} f(\frac{1}{v})$. The change of coordinates is then defined by $(t,s)  \mapsto \left(\frac{1}{v} ,\frac{u}{v^n}\right)$. 
\end{re}

\begin{re}[\textit{Superelliptic plane curves}] 
An immediate observation from the two chart cover constructed in \ref{sup_affine_model} is that the (singular) affine models for super elliptic curves are naturally embedded in the weighted projective plane $\PP(1,1,n)$. 

Furthermore, in the special case where $f$ has no multiple roots and $\kappa = 0$ or $\kappa = 1$, we have that the smooth superelliptic curve itself is a sub-variety of $\PP(1,1,n)$.
\end{re}

\begin{lem}[\textit{Equivalent affine models for superelliptic curves}] \label{lem_equivalence_sup}
The equivalence relation of superelliptic curves from definition \ref{superelliptic_def} is extended to affine models by the relations
\begin{enumerate}
\item $ s^2 - f(t) \, \sim\, s^2 - (ct+d)^{2g+2} f\big(\frac{at+b}{ct+d}\big)$, for  $\binom{a~b}{c~d}\in \mathrm{GL}(2,\C)$;
\item $s^p - \prod_{j} (t-a_{j})^{k_j} \,\sim \,s^p - \prod_{j} (t-a_{j})^{\zeta\cdot a_j}$, for $\zeta \in \FF_p^*$ and where $\zeta\cdot a_j$ is multiplication in $\FF_p$. 
\end{enumerate}
\end{lem}
\begin{proof}
Property (i) is straight-forward and also arises in the hyperelliptic case, hence we consider only property (ii). For this we define $\eta:C\rightarrow \PP^1$ and $\eta':C'\rightarrow \PP^1$ to be the superelliptic curves defined by $s^p - \prod_{j} (t-a_{j})^{k_j}$ and $s^p - \prod_{j} (t-a_{j})^{\zeta\cdot a_j}$ where $p$ divides the sum $\sum_j k_j$. Also, since $k_i \in \FF_p^*$ for each $j$ as well as $\zeta \in \FF_p^*$, we can assume that $k_1=1$.  

We begin by considering the monodromies on $C$ arising from the ramification points. Let $x\in\PP^1$ be distinct from the branch points and consider $m$ different loops on $\PP^1$ starting and finishing at $x$. The constraint for the loops is that each contains exactly one branch point and different loops contain different branch points. 

The monodromy arising from the preimages of the loop around $a_i$ can be described by considering the equation $s^p =r$ while considering the preimage of the unit circle and identifying $x$ with $1$. In this case, the preimage of the unit circle is
\[
\left\{ \left( \left. e^{\tau\, 2\pi i }, e^{\frac{\tau}{p} \,  2\pi i} \right) ~ \right|~ t \in [0, p]  \right\}. 
\]
The points in the preimage of $1$ are the $p$-th roots of unity and we label them by their power of the primitive root of unity $e^{\frac{2\pi i}{p} }$. Then the associated permutation from the monodromy representation is the  $p$-cycle
\[
(1, 2, \ldots, (p-1)).
\]
Similarly, for the other branch points we consider the equations $s^p =r^{k_j}$. In this case, the associated permutation from the monodromy representation is the  $p$-cycle 
\[
(k_i, 2\cdot k_i, \ldots, (p-1)\cdot k_i).
\]

Now we apply the same process on $C'$ beginning with $\zeta\cdot k_1 = \zeta$. In this case, the preimage of the unit circle is
\[
\left\{ \left( \left. e^{\tau\, 2\pi i }, e^{\frac{\tau \zeta }{p} \,  2\pi i} \right) ~ \right|~ t \in [0, p]  \right\}
\]
and we again label the points in the preimage of $1$ by their power of $e^{\frac{2\pi i}{p} }$. In this way, the ramification monodromy on $C'$ is described by the permutations
\[
(\zeta\cdot k_i, 2\cdot \zeta\cdot k_i, \ldots, (p-1)\cdot \zeta \cdot k_i).
\]

Now, define $\sigma$ to be the inverse of the permutation defined by $n \mapsto \zeta \cdot n$. This gives
\[
(\zeta\cdot k_i, 2\cdot \zeta\cdot k_i, \ldots, (p-1) \cdot\zeta\cdot k_i)\,\sigma  = ( k_i, 2 \cdot k_i, \ldots, (p-1)\cdot k_i)
\]
and shows that the $\eta$ and $\eta'$ have the same monodromy data (only a different choice of labelling). The uniqueness part of the Riemann existence theorem now shows that $\eta$ and $\eta'$ must be equivalent. 
\end{proof}

\begin{rem}[\textit{Constructing superelliptic curves from points}]\label{re_superelliptic_curves_from_points}
A key aspect of the results from section \ref{background_hyperelliptic} was the observation (described in  \ref{construct_hyperelliptic_from_points}) that there is a unique hyperelliptic curve associated to a given choice of branch points. Among other things, this observation leads to the isomorphism $\Hyp_g \cong M_{0,2g+2}/S_{2g+2}$ from theorem \ref{isomorphism_hyperelliptic_moduli_scheme}. 

However, in the superelliptic case, things are more complicated. The affine model for superelliptic curves (described in \ref{sup_affine_model}) is of the form $s^p = f(t)$ and suggests that we must also take into consideration the root-orders of $f(t)$. Indeed, this is confirmed by Lemma \ref{lem_equivalence_sup}. 
\end{rem}

\begin{re}[\textit{Notation for indexing components of $\Sup^p_g$}] \label{superelliptic_indexing_convention}
We will index the possible choices of affine model $s^p = f(t)$ by indexing the possible combinations for the root-orders of $f(t)$. In this light, it makes sense to group the ramification points with same root orders. If we assume that the representative of the superelliptic curve is \textit{not} branched at $\infty\in \PP^1$ then this allows us  to express the affine model in the form
\[
0= s^p - \prod_{k=1}^{p-1} \prod_{i=1}^{m_k} (t-\alpha_{k,i})^{k}. 
\]
An initial observation from this is that the vector $(m_1,\ldots, m_{p-1})$ has the properties:
\[
\mbox{(i) }~ \sum_{k=1}^{p-1} m_k  = m
\hspace{0.75cm}
\mbox{and}
\hspace{0.75cm}
\mbox{(ii) }~ \sum_{k=1}^{p-1}k\, m_k  \equiv 0 ~(\mathrm{mod}~ p).
\]
Moreover, the equivalence relation from Lemma \ref{lem_equivalence_sup} shows that for $\zeta \in \F_p^*$, we have that $(m_1,\ldots, m_{p-1})$ and $( m_{\zeta\cdot 1}, \ldots,m_{\zeta\cdot (p-1)})$ give equivalent superelliptic curves.  On top of this, there always exists a unique choice of $\zeta \in \F_p^*$ such that the first entry is $m_1 = \mathrm{max}\,\{m_1,\ldots,m_{p-1}\}$. This leads us to define the following indexing set
\[
\mathsf{M} = 
\left\{ (m_1, \ldots, m_{p-1}) \in (\ZZ_{\geq 0} )^{p-1}
\left|\,
\begin{array}{l}
\bullet~ \sum\limits_{k=1}^{p-1} m_k  = m, \\\bullet~ \sum\limits_{k=1}^{p-1}k\, m_k  \equiv 0 ~(\mathrm{mod}~ p), \mbox{ and} \\
\bullet~m_1=\mathrm{max}\,\{m_i\}. \textcolor{white}{\Big|}
\end{array} \right.\right\}.
\]
\end{re}

\begin{re}[\textit{Group describing equivalent affine models}] \label{re_group_equivalent_affine_models}
As was the case for hyperelliptic curve, we have to account for the inherent labelling of ramification points that an affine model gives. In the hyperelliptic case, all the ramification points were the same type, so this choice was accounted for by taking the $S_{2g+2}$ quotient. 

However, things are more complicated in the superelliptic case. To begin with, for each $\mathsf{m}\in \mathsf{M}$ the group 
\[
S_{\mathsf{m}} := S_{m_1} \oplus \cdots  \oplus  S_{m_{p-1}} 
\]
accounts for permutations of the ramification points within the groupings that correspond to $\mathsf{m}$. 

Another complication is that each $\mathsf{m}\in \mathsf{M}$ may define equivalent superelliptic curves in multiple ways. This arises from the equivalence relations described in Lemma~\ref{lem_equivalence_sup} part (iii). The lemma shows that there is a natural action of  $\F_p^*$ on the $(p-1)$-tuples $(m_1,\ldots, m_{p-1})$ which is defined by 
\[
\zeta\cdot (m_1,\ldots, m_{p-1}) := ( m_{\zeta\cdot 1}, \ldots,m_{\zeta\cdot (p-1)}).
\]
Moreover, this action preserves properties (i) and (ii) from \ref{re_superelliptic_curves_from_points}. Hence we consider the subgroup of $\FF_p^*$ which stabilises a given $\mathsf{m} \in \mathsf{M}$ and denoted it by
\[
\mathrm{Stab}_p(\mathsf{m}) :=  \big\{~ \zeta\in \FF_p^* ~~\big|~~  \zeta \cdot \mathsf{m} = \mathsf{m} ~\big\}.
\]

Both $S_{\mathsf{m}}$ and $\mathrm{Stab}_p(\mathsf{m})$ are naturally subgroups of $S_m$. Moreover, both of these inclusions arise from the fact that for any given $\mathsf{m}\in \mathsf{M}$ with $\mathsf{m} = (m_1,\ldots,m_{p-1})$, we can uniquely describe each $n\in \{1,\ldots, m\}$ by a pair $(i,k)$ where $i\in \{1,\ldots, p-1\}$ and $ 0< k\leq m_i$. Explicitly, this relationship is described by
\[
n =~ k~ + \sum_{0 < j < i} m_j.
\]
Using this relationship can now write:
\begin{enumerate}
\item The inclusion $\delta: S_{\mathsf{m}}\hookrightarrow S_m$ is defined by
\[
\delta(\sigma_1,\ldots, \sigma_{p-1}) \Big(k + \sum_{0 < j < i} m_j \Big) := ~\sigma_j(k) ~+ \sum_{0 < j < i} m_j. 
\]
\item The inclusion  $\gamma: \mathrm{Stab}_p(\mathsf{m}) \hookrightarrow S_m$ is defined by
\[
\gamma(\zeta) \Big(k + \sum_{0 < j < i} m_j \Big) := ~k ~ + \sum_{0 < j < \zeta\cdot i} m_j. 
\]
\end{enumerate}

We observe that as subgroups of $S_m$, we have that the intersection $S_{\mathsf{m}} \cap \mathrm{Stab}_p(\mathsf{m})$ is trivial and that $\zeta \cdot \sigma = \sigma \cdot \zeta$ for all $\zeta \in \mathrm{Stab}_p(\mathsf{m})$ and $\sigma\in S_{\mathsf{m}}$. Hence, we also have:
\[
\mathsf{A}_{\mathsf{m}} \,:= ~S_{\mathsf{m}} \cdot\mathrm{Stab}_p(\mathsf{m})  ~\,\cong~ S_{\mathsf{m}} \oplus \mathrm{Stab}_p(\mathsf{m}).
\]
The group $\mathsf{A}_{\mathsf{m}}$ describes the equivalent ways that $\mathsf{m}\in\mathsf{M}$ can give rise to a superelliptic curve. This results in Proposition \ref{decomposition_of_Sup}. 
\end{re}

\begin{prop}[\textit{Decomposition of $\Sup^p_g$ into connected components}]\label{decomposition_of_Sup}
The set $\mathsf{M}$ defined in \ref{superelliptic_indexing_convention} indexes the connected components of $\Sup^p_g$ to give
\[
 \Sup^p_g = \bigsqcup_{\mathsf{m} \in \mathsf{M}}  \Sup^p_{g,\mathsf{m}}.
\]
For each $\mathsf{m} = (m_1,\ldots, m_{p-1}) \in \mathsf{M}$ there is an isomorphism $M_{0,m}/\mathsf{A}_{\mathsf{m}}  \cong \Sup^p_{g,\mathsf{m}} 
$ where $\mathsf{A}_{\mathsf{m}}$ is the group defined in \ref{re_group_equivalent_affine_models}. The $\mathsf{A}_{\mathsf{m}}$-quotient map composed with this isomorphism gives a morphism 
\[
\phi_{\mathsf{m}} : M_{0,m} \longrightarrow  \Sup^p_{g,\mathsf{m}}
\]
which is defined by mapping the equivalence class $\big[\,[a_1:b_1],\ldots, [a_m:b_m]\,\big]\in M_{0,m}$ to the equivalence class of superelliptic curves with the representative
\[
0= s^p - \prod_{k=1}^{p-1} \prod_{i=1}^{m_i} (a_{\varphi(k,i)} t-b_{\varphi(k,i)})^{k}, 
\]
where $\varphi(k,i) := k + \sum_{0 < j < i} m_j$. 
\end{prop}

\begin{rem}[\textit{Different equivalence conditions}]
If we changed the equivalence relation between superelliptic curves from  definition \ref{superelliptic_def} to be the stricter condition: 
\textit{ $\pi:C\rightarrow \PP^1$ and $\pi':C'\rightarrow \PP^1$ are equivalent if there is an isomorphisms $\alpha:C\rightarrow C'$ such that $\pi \circ \alpha = \pi'$.} Then much of this section would still apply except with different equivalence conditions. 

In this case, the resulting construction in Proposition~\ref{decomposition_of_Sup} would  have been carried out using the configuration space $\mathrm{Conf}^m(\PP^1)$ in place of $M_{0,m}$. Moreover, this construction would result in a natural sub-moduli-space of the moduli space of maps $M_g(\PP^1, p)$. The resulting sub-moduli-space has links to $(p-1)$-spin Hurwitz theory and was studied by the second author in \cite{leigh_div_ram} and \cite{leigh_loc}. 
\end{rem}

\begin{rem}[\textit{Codimension in $M_g$}]
An immediate observation from Corollary~\ref{decomposition_of_Sup} is that superelliptic curves are rare among genus $g$ curves. The image of $\Sup_g^p \rightarrow M_g$ has codimension $3g-3 - (m-3)  = 3g-1 - \frac{2g}{p-1}$ which we note has a maximum value of $3g-2$ that occurs when $2g=p-1$.  
\end{rem}

\section{Divisors on superelliptic curves}\label{divisors_section}

\begin{re}[\textit{Notation for this section}]\label{divisor_section_notation}
In this section we will consider a fixed superelliptic curve $\pi: C\rightarrow \PP^1$ with $g>0$ and a choice of affine model with the same notation given in \ref{sup_affine_model}. Namely, $\pi$ will be given by the equations
\[
s^p = f(t) 
\quad \mbox{and} \quad 
u^p = v^{\kappa} h(v)
\]
where $f,h\in \CC[t]$ such that for $n=:\lceil \deg(f)/p \rceil$, $\kappa := pn - \deg(f)$ we have that $h$ is the unique polynomial with $v^{\kappa} h(v) = f(\frac{1}{v})$. The change of coordinates between the two charts is then defined by $(t,s)  \mapsto \left(\frac{1}{v} ,\frac{u}{v^n}\right)$.

Moreover, in this section, we will assume that these coordinates have been chosen such that $\pi$ is has a branch point at $\infty \in \PP^1$. This condition is equivalent to requiring that $p \nmid \deg(f)$ or equivalently requiring that  $\kappa \in \FF_p^*$. 

Lastly, we will assume an ordering for the branch points in $\PP^1\setminus\{\infty\}$. With this ordering assumed we will write $f(t) = \prod_i (t - a_i)^{k_i}$ where we have also invoked the equivalence described in Lemma~\ref{lem_equivalence_sup} to make this a monic polynomial. 
\end{re}

\begin{re}[\textit{Points at infinity and branch points}]
We will follow the standard convention and refer to the preimage $\pi^{-1}(\infty)$ as \textit{the point at infinity} and denote it by $Q_{\infty}$. In the notation from \ref{divisor_section_notation} this corresponds to the point defined by $v=u=0$ in the second chart. 
\end{re}

\begin{re} [\textit{Natural principal divisors and the superelliptic divisor class}]\label{natural_principal_superelliptic_divisors}
Recall from \ref{divisor_section_notation} that  $f(t) = \prod_{i} (t-a_{i})^{k_i}$ and observe that $pn =\kappa +  \sum_i k_i$. Then we have the natural principal divisors given by:
\begin{enumerate}
\item \textit{Horizontal:} $(s) = \left(\tfrac{u}{v^n}\right) = \sum_i k_i (Q_i  -  Q_\infty ). $
\item \textit{Vertical ramified:} $(t-a_i) = \left(\tfrac{1-a_i u}{u}\right) = p Q_i  -  p Q_\infty. $
\item \textit{Vertical unramified:} $(t-a) = \left(\tfrac{1-a u}{u}\right) = \pi^{*}P_a  -  p Q_\infty $.
\end{enumerate}
Here $a\in \CC^*\setminus\{a_1,\ldots,a_{m-1}\}$ and $P_a\in\PP^1$ is the corresponding divisor. The divisor $\pi^{*}P_a$ consists of $p$ distinct points. 

\end{re}

\begin{rem}
In \ref{natural_principal_superelliptic_divisors} and in what follows, we will always mean the divisor on the smooth curve $C$ and not the (potentially) singular affine model.  
\end{rem}

\begin{re}[\textit{$p$-torsion in the Jacobian of a superelliptic curve}]\label{re_Delta_defintion}
A key focus of this article is to study $p$-torsion within the Jacobian of a superelliptic curve
(for related, but somewhat orthogonal, investigations of this topic, see Arul's thesis \cite{arul}).
An immediate observation from \ref{natural_principal_superelliptic_divisors} is that
\[
p\cdot (Q_i  -  Q_\infty) \sim 0
\]
for each $i$. This shows that $[Q_i - Q_\infty]$ is an element of $\Jac(C)[p]$ 
(c.f. the construction in Section~\ref{background_hyperelliptic}). 
We can now consider the $\FF_p$-vector subspace of $\Jac(C)[p]$ spanned by the classes $D_i := [Q_i - Q_\infty]$ and denote it by
\[
\Delta := \FF_p\mbox{-}\mathrm{Span}\Big\{ D_1, \ldots, D_{m-1}\Big\} \subseteq \Jac(C)[p].
\]
We also have the relationship
\[
\sum_i k_i D_i  = \sum_i k_i [Q_i - Q_\infty] = \big[ (s) \big] = 0
\]
which shows that $D_1, \ldots, D_{m-2}$ is a $\FF_3$-spanning set for $\Delta$. In fact, by the following theorem, this turns out to be a basis for $\Delta$. 
\end{re}

\begin{thm}[\textit{A $\FF_p$-basis for $\Delta$, \cite[Prop. 6.1]{poonenshaefer} \& \cite[Thm. 1]{wawrow}}] \label{thm_delta_basis}  
 Let $\Delta$ be the subgroup of $\jac(C)[p]$ generated by classes of the form  $D_i=  [Q_i  - Q_\infty]$. Then $\Delta$ is isomorphic to $\FF_p^{m-2}$ and $D_1,\ldots, D_{m-2}$ is a $\FF_p$-basis for $\Delta$. 
\end{thm}

\begin{rem} 
An immediate observation from the construction in \ref{re_Delta_defintion} is that for $p >2 $ we  will not get all of $\jac(C)[p] \cong \FF_p^{2g}$. Indeed, when $2g=p-1$, the codimension of $\Delta$ in $\Jac(C)[p]$ is $2g-1$.  
\end{rem}

\begin{re}[\textit{Weil pairing on superelliptic curves}] 
Let $[E]$ and $[E']$ be two elements of $\jac(C)[p]$ and let $E \in [E]$ and $E' \in [E']$ be divisors
with disjoint support. Let $f$ and $g$ be functions on $C$ such that
\[
pE = (f) \quad \text{and} \quad pE'=(g).
\]
The Weil pairing of $[E]$ and $[E']$ is then defined as the quotient
\[
w([E],[E']):=\frac{f(E')}{g(E)} \in \mu_p
\]
where $f(E')$ is defined by $
f(E') = \prod_{P \in E'} f(P)^{\mathrm{mult}_P(E')}
$ and $g(E)$ is analogously  defined. 
The Weil pairing is a symplectic pairing on $\jac(C)[p]$.
For more details and relations to moduli, see e.g. \cite[Appx. B]{arbarellocgh} and
\cite[Sec. 7.2]{mumfordfogartykirwan}.
\end{re}

\begin{prop}\label{weil_pairing_prop}
When $p$ is odd the vector space $\Delta$ is an isotropic subspace of $\jac(C)[p]$ with respect to the Weil pairing.  In other words, for any two divisor classes $D, D'$ in $\jac(C)[p]$ we have that $w(D, D') = 1$. 

If $p$ is even and $D_i$ is the divisor class defined in \ref{re_Delta_defintion} then $w(D_i, D_j) = -1$ for $i\neq j$. 
\end{prop}

\begin{proof}
We will use the notation from \ref{divisor_section_notation}. Namely that $\pi:C\rightarrow\PP^1$ is given by the polynomials $s^p = f(t)$ and $u = v^{\kappa}g(v)$. Furthermore, using 
Lemma~\ref{lem_equivalence_sup}, we may assume that $\kappa = 1$ and that $f(t) = \prod_i (t - a_i)^{k_i}$ such that $a_i \in \CC^*$ are non-zero and distinct.

From the definition in \ref{re_Delta_defintion}, we have that $\Delta$ is generated by classes with representatives of the form
 \begin{equation*}
  E_i = Q_i-Q_{\infty}, \quad i=1, \ldots, m-1.
 \end{equation*}
 Moreover, since the Weil pairing is alternating we have that $w([E_i],[E_i])=1$, it will suffice to show that $w([E_i],[E_j])=1$ for $i\neq j$. Hence, we assume that $i\neq j$. 
 
Since $E_i$ and $E_j$ do not have disjoint support we need to find a divisor $F_j$ in the class $E_j$  such that $F_j$ has support disjoint from $E_i$. The construction of the divisor $F_j$ is a key aspect of this proof.
 
To construct $F_j$ we consider the rational function $(t^{n-1}-s)/t^n \in k(C)^*$ and note that it corresponds to $u-v$ in the other chart (where $C$ is defined by the equation $u^p - v\, g(v)=0$). To determine the divisor $(u-v)$, we consider the isomorphism 
\[
\CC[u,v] /( v g(v) -u^p , u-v) \cong \CC[v]/ \big(v( g(v) - v^{p-1}  )\big)
\]
and  recall that $g(v) = \prod_i (1 - a_i v)^{k_i}$. Then we consider the two unique factorisations
\vspace{-0.4em}\begin{align}
g(v) - v^{p-1} = \prod_\alpha(1-b_\alpha v)^{\lambda_\alpha}
\hspace{0.5cm}
\mbox{and}
\hspace{0.5cm}
f(t) - t^{p(n-1)} =  \prod_\alpha(t-b_\alpha )^{\lambda_\alpha}. 
\label{b_alpha__defining_equation}
\end{align}
which  are equivalent under the change of variables $u \mapsto \frac{1}{t}$. We then denote by $G_\alpha$ the divisors on $C$ given by the points $(t,s) = (b_{\alpha},b_{\alpha})$. Then the principal divisor $(u-v)$ is given by
\[
(u-v) =  Q_\infty + \sum_\alpha \lambda_\alpha G_\alpha - n \cdot \sum_\beta S_\beta. 
\]
where $\sum S_\beta$  is the divisor of the $p$ disjoint points $\pi^{-1}(0)$.

Now we can define $F_j$ by
\[
F_j := E_j + (u-v) = Q_j + \sum_\alpha \lambda_\alpha G_\alpha - n \cdot \sum_\beta S_\beta
\]
and define the rational function $\psi \in \CC(C)^*$ by $\psi= (t-a_j)(t^{n-1}-s)^p/t^{pn}$, while defining the  rational function $\varphi \in \CC(C)^*$ by $\varphi=(t-a_i)$.  We note that by construction we have the properties required to evaluate the Weil pairing. Namely, we have that
\begin{enumerate}
\item $E_i \cap F_j = \emptyset$,
\item $(\phi) = p\cdot E_i$ and
\item $(\psi) = p\cdot F_j$.
\end{enumerate}

We will evaluate the numerator and denominator of $w([E_i],[F_j])$ individually beginning with $\psi(E_i)$. First note that in the  second chart we have that the rational function $\psi$ is given by $(-1)^p (1-a_j v)(u-v)^p/v $ and recall that $C$ is defined in this chart by the equation $u^p - v \,g(v)=0$. Now the binomial theorem shows that the rational function is given by
\begin{align*}
\psi
     &=
 (1-a_j v) \left((-1)^p g(v)   +   \sum_{\gamma=1}^p\binom{p}{\gamma} (-u)^{p-\gamma}  v^{\gamma-1}\right).
\end{align*}
In this chart we have that $Q_i$ is given by $(v,u) = (\frac{1}{a_i}, 0)$ and $Q_\infty$ is given by $(v,u) = (0,0)$, hence we can evaluate $\psi(E_i)$ as:
\begin{align}
 \psi(E_i) = \frac{\psi(Q_i)}{\psi(Q_\infty)} = \frac{(1-a_j/a_i) a_i^{1-p} }{ (-1)^p} = \frac{a_i - a_j}{(- a_i)^p}.
\label{w_psi_calc}
\end{align}

Now we will evaluate $\varphi$ at $F_j$ after recalling that $Q_j$ corresponds to $(t,s) = (a_j,0)$, $G_\alpha$ corresponds to $(t,s) = (b_{\alpha},b_{\alpha})$ and $S_\beta$ corresponds to pairs with $t=0$. Hence we have
\begin{align*}
\varphi(F_j) 
=
 \varphi(Q_j)\cdot \prod_\beta  \varphi(G_\alpha)^{\lambda_\alpha} \cdot \prod_\alpha  \varphi(S_\beta)^{-n}
 =
 \frac{(a_j-a_i) \prod_\alpha (b_\alpha - a_i)^{\lambda_\alpha} }{(-a_i)^{pn}}.
\end{align*} 
Now considering equation (\ref{b_alpha__defining_equation}) in the $t$ coordinate and noting that $f(a_i)=0$ and $\deg(f)=pn-1$ we have the equations
\begin{align}
 (a_i)^{-p} = \frac{\prod_\alpha(b_\alpha-a_j)^{\lambda_\alpha}}{(-a_i)^{pn}}
 \hspace{0.5cm}
 \mbox{and}
 \hspace{0.5cm}
 \varphi(F_j) 
=
 \frac{-(a_i-a_j) }{(a_i)^{p}}.
 \label{w_phi_calc}
\end{align}

Finally, combining equations (\ref{w_psi_calc}) and (\ref{w_phi_calc}), we have
\[
w([E_i],[F_j]) 
=
 \frac{ \psi(E_i) }{\varphi(F_j) }
 = (-1)^{p-1}
\]
which is $-1$ when $p$ is even and $1$ when $p$ is odd. 
\end{proof}

\begin{cor} \label{cor_maximal_isotopic_Delta}
 The vector space $\Delta$ is a maximal isotropic subspace of $\jac(C)[3]$ with respect to
 the Weil pairing.
\end{cor}
\begin{proof}
 Any maximal isotropic subspace of $\jac(C)[3]$ has dimension $g$ and, conversely,
 any isotropic subspace of $\jac(C)[3]$ of dimension $g$ is a maximal isotropic
 subspace. Thus, since $\Delta$ has dimension $g$ the result needs follows immediately from 
 Proposition~\ref{weil_pairing_prop}. 
\end{proof}

\section{Trigonal Superelliptic Curves} \label{trig_sec}

\begin{re}[\textit{Overview of this section}] 
The case $p=3$ is the natural next step from the hyperelliptic case. In this case, superelliptic curves exist for each $g>0$ and we will see that the key decomposition result of the moduli space of hyperelliptic curves with level-2 structure from theorem \ref{Hyperelliptic_w_level_structure_decomposition_theorem} has an analogue in the $p=3$ case.
\end{re}

\begin{re}[\textit{Properties of the case $p=3$}]\label{re_p=3_specifics}
As discussed in \ref{re_superellitpic_ramification}, trigonal superelliptic curves exist for every genus $g>0$ and the number of ramification points is
\[
m = g+2. 
\]
In this case we have $\FF_3^* \cong S_2$ and the indexing set $\mathsf{M}$ from \ref{superelliptic_indexing_convention} contains ordered pairs $(m_1, m_2)$. Explicitly, the indexing set can be expressed as
\[
\mathsf{M} = \left\{ \big(m - 3i - r,\, 3i +r \big) \left| 
\begin{array}{l}
\mbox{$r\in\{0,1,2\}$ with $m = 3n - r$ for some $n\in \ZZ$ and}\\
\mbox{$i\in\ZZ$ with $0\leq i \leq (m-2r)/6$.}
 \end{array}
 \right.
\hspace{-0.25em} \right\}.
\]
Some examples of the indexing set are given in below in table \ref{table_indexing_set_p=3}. We can also examine the group $\mathsf{A}_{\mathsf{m}}$ describing equivalent affine models from \ref{re_group_equivalent_affine_models}. There are two cases:
\begin{enumerate}
\item $\mathsf{m} = (a, a)$: This case occurs whenever $m$ is even and in this case we have $\mathrm{Stab}_p(\mathrm{m}) = S_2$ and $\mathsf{A}_{\mathsf{m}} \cong S_2 \cdot (S_a \oplus S_a) \cong S_2 \oplus S_a \oplus S_a$. In terms of the explicit description of $\mathsf{M}$ given above, this occurs when $i = (m-2r)/6$. 
\item $\mathsf{m} = (a, b)$ where $a > b$: In this case we have that $\mathrm{Stab}_p(\mathrm{m})$ is trivial  and hence $\mathsf{A}_{\mathsf{m}} \cong S_a \oplus S_b$.
\end{enumerate}
\end{re}

\begin{tableRem} \label{table_indexing_set_p=3} 
Examples of the indexing set $\mathsf{M}$ for $p=3$. The values of $\mathsf{m}$ for which $S_{\mathsf{m}} \neq S_{m_1} \oplus S_{m_2}$ are shown in bold.

\noindent 
\begin{tabu}{X[c]}
\begin{tabu}{|X[c]|X[c]|X[c5]| X[0.5] |X[c]|X[c]|X[c5]|}
\cline{1-3} \cline{5-7}
$g$ & $m$ &$\mathsf{M}$ && $g$ & $m$ & $\mathsf{M}$ \\
\cline{1-3} \cline{5-7}
$1$ &$3$ & $\{(3,0)\}$ && $7$ &$9$ & $\{(9,0), (6,3) \}$ \\
$2$ &$4$ & $\{\bm{(2,2)}\}$ && $8$ &$10$ & $\{(8,2), \bm{(5,5)} \}$ \\
$3$ &$5$ & $\{(4,1)\}$ && $9$ &$11$ & $\{(10,1), (7,4) \}$ \\
$4$ &$6$ & $\{(6,0), \bm{(3,3)} \}$ && $10$ &$12$ & $\{(12,0), (9,3), \bm{(6,6)} \}$ \\
$5$ &$7$ & $\{(5,2) \}$ && $11$ &$13$ & $\{(11,2), (8,5) \}$ \\
$6$ &$8$ & $\{(7,1), \bm{(4,4)} \}$ && $12$ &$14$ & $\{(13,1), (10,4), \bm{(7,7)} \}$ \\
\cline{1-3} \cline{5-7}
\end{tabu}\\[-0.95em]
\mbox{}\\
\end{tabu}

\end{tableRem}
 
\begin{re}[\textit{Level 3 structures and moduli description of $\Sup^{3}_{g}[3]$}] 
Here we recall that the definition of a superelliptic curve with a level-3 structure is a pair 
\[
\Big(~~ \pi:C\rightarrow \PP^1, ~~\eta: \FF_3^{2g} \overset{\cong}{\longrightarrow} \Jac(C)[3]   ~~\Big)
\]
where $\pi$ is a superelliptic curve and $\eta$ is an isometry from $\FF_3^{2g}$ with the standard symplectic form to $\Jac(C)[3]$ with the Weil pairing. In other words, a level-3 structure is a choice of an ordered symplectic basis for $\Jac(C)[3]$ compatible with the Weil paring. 

The moduli space of $p=3$ superelliptic elliptic curves with level $3$ structure is denoted by $\Sup^{3}_{g}[3]$ and comes with a natural morphism
\[
\mathsf{F}: \Sup^{3}_{g}[3] \longrightarrow \Sup^{3}_{g} 
\]
which forgets the level $3$ structure.  The group $\Sp(2g,\FF_3)$ acts naturally on $\Sup^{3}_{g}[3]$ by changing basis and  has no fixed points. The morphism $\mathsf{F}$ is then equivalent to taking the quotient by $\Sp(2g,\FF_3)$ which shows that $\mathsf{F}$ is étale of degree $|\Sp(2g,\FF_3)|$.

For each $\mathsf{m}\in\mathsf{M}$, we can extend this concept to the sub-moduli space $\Sup_{g,\mathsf{m}}^3$ (defined via the decomposition in proposition \ref{decomposition_of_Sup}) by the following diagram where all squares are cartesian.
\[
\begin{tikzcd}
 \Sup_{g,\mathsf{m}}^3[3] \arrow[d, "\mathsf{F}_{\mathsf{m}}"] \arrow[r,hook]
& \Sup_{g}^3[3] \arrow[d, "\mathsf{F}"] \\
 \Sup_{g,\mathsf{m}}^3 \arrow[r, hook]
& \Sup_{g}^3
\end{tikzcd} 
\]

\end{re}

 \begin{re}[\textit{A natural basis for $\Jac(C)[3]$}] \label{extending_basis_for_3_torsion}
 
Consider a moduli point $[P_1,\ldots, P_m] \in M_{0,m}$  and let $[\pi:C\rightarrow\PP^1]$ be the associated point  in $\Sup_{g,\mathsf{m}}^3$. Recall from section \ref{divisors_section}  that  the divisor classes
\[
D_j := [Q_j  - Q_m] 
\]
for $j\in\{1,\ldots,g\}$ form a basis for a natural  subspace  $\Delta \subset \Jac(C)[3]$. It was further shown in Corollary~\ref{cor_maximal_isotopic_Delta} that $\Delta$ is a Lagrangian (a maximal isotropic subspace) and hence that $\Delta$ is a natural choice of polarisation for the pair $(\Jac(C)[3], w)$ where $w$ is the Weil pairing. 
 
 Moreover, using standard techniques, we can use the basis $(D_1, \ldots, D_g)$ to define a basis for the complimentary isotropic subspace $\Delta^c\subset\Jac(C)[3]$. In particular, for each $j\in\{1,\ldots,g\}$ the conditions
\begin{equation*}
 w(D_k,E)=1, \quad k \neq j, \quad E \in \Delta^c
\end{equation*}
define $1$-dimensional subspaces $\Lambda_j \subset \Delta^c$. On top of this, the condition 
\begin{equation*}
 w(D_j,E_j) = e^{\frac{2\pi}{3} i}
\end{equation*}
uniquely defines an element $E_j \in \Lambda_j$. So, by construction, we have a natural symplectic basis for $\jac(C)[3]$ given by $(D_1, \ldots, D_g, E_1, \ldots, E_g)$.
\end{re}

 \begin{re}[\textit{Natural morphisms to $\Sup^{3}_{g,\mathsf{m}}[3]$}] \label{re_morphism_M_Sup}
 Continuing from the situation in \ref{extending_basis_for_3_torsion}, we take a moduli point $[P_1,\ldots, P_m] \in M_{0,m}$ and let $[\pi:C\rightarrow\PP^1]$ be the associated point  in $\Sup_{g,\mathsf{m}}^3$. Using the basis from \ref{extending_basis_for_3_torsion}, we can construct an isometry
 \[
 \eta: \FF_3^{2g} \overset{\cong}{\longrightarrow} \Jac(C)[3]
 \]
 by sending the standard basis $(e_1,\ldots,e_g,f_1,\ldots, f_g)$ for $\FF_3^{2g}$ to the natural basis  $(D_1, \ldots, D_g, E_1, \ldots, E_g)$ for $\Jac(C)[3]$. This then gives rise to a natural morphism 
 \[
  \psi_{\mathrm{id}}: M_{0,m} \longrightarrow \Sup^{3}_{g,\mathsf{m}}[3]
 \]
 defined by mapping the equivalence class $[P_1,\ldots, P_m] \in M_{0,m}$ to the equivalence class $[\pi, \eta]$. Here the subscript $\mathrm{id}$ refers to the identity in $ \mathrm{Sp}(2g, \FF_3)$.

 We can also extend this concept to any $A \in \mathrm{Sp}(2g, \FF_3)$ by first considering the automorphism  $\theta_{A}: \Sup^{3}_{g,\mathsf{m}}[3] \rightarrow  \Sup^{3}_{g,\mathsf{m}}[3]$ defined by $[\pi, \eta]\mapsto [\pi, T_A\circ \eta]$, where $T_A: \FF_3\rightarrow \FF_3$ is the linear change of basis  associated to $A$. The morphism  $ \theta_A \circ \psi_{\mathrm{id}}$  is denoted by
   \[
 \psi_A : M_{0,m} \longrightarrow \Sup^{3}_{g,\mathsf{m}}[3].
 \]
 
Combining this concept with the forgetful morphism and the isomorphism from 
Proposition~\ref{decomposition_of_Sup} gives the following commutative diagram.
\begin{equation}
\begin{array}{c}
\begin{tikzcd}
M_{0,m} \arrow[r, "\psi_A"] \arrow[d]
& \Sup^{3}_{g,\mathsf{m}}[3] \arrow[d, "{\mathsf{F}_{\mathsf{m}}}"] \\
M_{0,m}/A_{\mathsf{m}} \arrow[r, "\cong"]
& \Sup^{3}_{g,\mathsf{m}}
\end{tikzcd} 
\end{array} \label{sup_w_level_3_commuting_diagram}
\end{equation}
 
\end{re}

\begin{re}[\textit{Natural group homomorphism defined by $\psi_A$}] \label{re_inclusion_Am_into_Smyplectic}

For each $\mathsf{m} = (m_1,m_2) \in\mathsf{M}$, $P\in M_{0,m}$ and 
$A\in \mathrm{Sp}(2g,\FF_3)$, the morphism $\psi_A: M_{0,m} \rightarrow \Sup^3_g[3]$ from \ref{extending_basis_for_3_torsion} defines a group homomorphism which we will denote by 
\[
\Psi_{P,A} : \mathsf{A}_{\mathsf{m}} \rightarrow \Sp(2g,\FF_3).
\]
This homomorphism can be  describe explicitly by initially considering the case $A= \mathrm{id}$.  To begin, let $[P_1,\ldots,P_m]\in M_{0,m}$ and  $(\pi:C\rightarrow \PP^1, \eta)$ be the superelliptic curve and the level-3 structure which are given by $\psi_{\mathrm{id}}$. 

The homomorphism $\Psi_{P,\mathrm{id}}$ is constructed by considering the action of $S_{\mathsf{m}}$ on $M_{0,m}$ and showing that this defines a change of basis for $\Jac(C)[3]$. A straight-forward way of accomplishing this is to consider how the generators of $\mathsf{A}_{\mathsf{m}}$ define change of basis matrices for $\Jac(C)[3]$. 

First consider the generators of $S_{\mathsf{m}}= S_{m_1}\oplus S_{m_2}$ while considering $S_{\mathsf{m}}$ as a subgroup of $S_m$. Recall that $S_{m}$ is generated by the $m-1$ transpositions $(i, i+1)$ for $i\in \{1,\ldots, m-1\}$. In the case at hand, $S_{\mathsf{m}}$ is generated by the collection $(S_{m_1}\oplus S_{m_2}) \cap \{ (i,i+1)\}$. We can now consider how these permutations define changes of basis:
\begin{enumerate}
\item \textit{For $i\notin  \{m-1, m-2\}$:} The change of basis $\Psi_{P,\mathrm{id}}(\sigma)$  is defined on each $D_j$ by considering the action of $S_{\mathsf{m}}$ on $\{Q_1,\ldots Q_m\}$. The result is
\[
\Psi_{P,\mathrm{id}}(\sigma) \cdot D_j = \left\{
\begin{array}{ll}
D_j & \mbox{for $j \notin \{i,i+1\}$,}\\
D_{i+1} & \mbox{for $j= i$,}\\
D_{i} & \mbox{for $j= i+1$,}
\end{array}
\right.
\]
which is the usual ``\textit{row-swap}'' elementary row-operation. 
\item \textit{For $\sigma := (g,g+1) = (m-2,m-1)$:}  The generator  $\sigma := (g,g+1)$ exists in the cases $m_2 =0$, $m_2=1$ and $m_2 >3$. In these cases we have
\[
\Psi_{P,\mathrm{id}}(\sigma) \cdot D_j= \left\{
\begin{array}{ll}
D_j & \mbox{for $j \neq g$,}\\
D_{m-1} & \mbox{for $j= g$.}
\end{array}
\right.
\]
To see the associated change of basis, we recall from \ref{natural_principal_superelliptic_divisors} that the natural horizontal principal divisor gives:
\begin{enumerate}
\item \textit{For $m_2 = 0$:} $ D_{m-1} \sim 2\sum_{i = 1}^{m -2} D_i$.
\item \textit{For $m_2 = 1$:} $ D_{m-1} \sim \sum_{i = 1}^{m -2} D_i$.
\item \textit{For $m_2 \geq 2 $:} $ D_{m-1} \sim \sum_{i = 1}^{m_1} D_i + 2\sum_{i=m_1+1}^{m-2} D_i $.
\end{enumerate}

\item \textit{For $\sigma := (g+1,g+2) = (m-1,m)$:}
The generator $\sigma := (g,g+1)$ exists in the cases $m_2 =0$ and $m_2 >2$. In these cases we have the same change of basis and relations as those in part (ii) except
\[
\Psi_{P,\mathrm{id}}(\sigma) \cdot D_j=  D_j + 2 \,D_{m-1}. 
\]
\end{enumerate}

We now consider the extra generators in the case when $\mathrm{Stab}_p(\mathsf{m}) \neq 0$. This case occurs when $\frac{m}{2}\in \ZZ_{>0}$ and $\mathsf{m} = (\frac{m}{2},\frac{m}{2})$  so that $\mathrm{Stab}_p(\mathsf{m}) \cong S_2$. Let $\zeta$ be the generator of $\mathrm{Stab}_p(\mathsf{m})$ then we have
\[
\Psi_{P,\mathrm{id}}(\zeta) \cdot D_j 
= \left\{
\begin{array}{ll}
D_{j+\frac{m}{2}} + 2\, D_{\frac{m}{2}} & \mbox{for $1\leq j < \tfrac{m}{2}-1$,}\\
D_{m-1} + 2\, D_{\frac{m}{2}} & \mbox{for $j = \tfrac{m}{2}-1$,}\\
 2\, D_{\frac{m}{2}} &  \mbox{for $j = \tfrac{m}{2}$,}\\
D_{j-\frac{m}{2}} + 2\, D_{\frac{m}{2}} & \mbox{for $\frac{m}{2}< j \leq m-2$,}
\end{array}
\right.
\]
where $D_{m-1}$ has been calculated above in (ii). 

Hence, we have defined a homomorphism $\Psi_{P,\mathrm{id}} : \mathsf{A}_{\mathsf{m}} \rightarrow \Sp(2g,\FF_3)$. Now, consider the automorphism $\Theta_{A}: \Sp(2g,\FF_3) \rightarrow \Sp(2g,\FF_3)$ which is defined by $ B\mapsto AB$. The group homomorphism  $\Psi_{P,A}$ is now defined by $\Psi_{P,A} := \Theta_A\circ \Psi_{P,\mathrm{id}}$. 
\end{re}

\begin{lem}\label{lem_psi_A_injective}
Let $A\in \Sp(2g,\FF_3)$ and $P\in M_{0,m}$. The morphism $\psi_A$  from \ref{re_morphism_M_Sup} and the group homomorphism $\Psi_{P,A}$ from \ref{re_inclusion_Am_into_Smyplectic} are injective. 
\end{lem}
\begin{proof}
Since for $A\neq \mathrm{id}$, the morphisms $\psi_A$ and $\Psi_A$ are defined by post-composing with automorphisms, we are only required to consider the case when $A=\mathrm{id}$. 

To consider whether $\psi_{\mathrm{id}}$ is injective take two points in $P,P'\in M_{0,m}$ which have $\psi_{\mathrm{id}}(P) = \psi_{\mathrm{id}}(P')$. We also let $D_1,\ldots, D_g$ and $D'_1,\ldots, D'_g$ be the associated natural bases constructed from divisors in \ref{re_Delta_defintion} and corresponding to the images $\psi_{\mathrm{id}}(P)$ and $\psi_{\mathrm{id}}(P')$. 

The combination of the condition $\psi_{\mathrm{id}}(P) = \psi_{\mathrm{id}}(P')$  and fact that the diagram (\ref{sup_w_level_3_commuting_diagram}) commutes shows that $P$ and $P'$ must be in the same $\mathsf{A}_{\mathsf{m}}$-class. Taking $\sigma\in \mathsf{A}_{\mathsf{m}}$ to be such that $P' = \sigma \cdot P$, we observe that this implies $D'_i = \Psi_{P,\mathrm{id}}(\sigma)\cdot D_i$ for all $i\in\{1,\ldots,g\}$. The condition $\psi_{\mathrm{id}}(P) = \psi_{\mathrm{id}}(P')$ is now only true if $D_i = \Psi_{P,{\mathrm{id}}}(\sigma)\cdot D_i$ for all $i\in\{1,\ldots,g\}$. Hence we have shown that $\psi_{\mathrm{id}}$ is injective only if 
$\Psi_{P,\mathrm{id}}$ is injective for all $P\in M_{0,m}$.

Let $\sigma \in \mathsf{A}_{\mathsf{m}}$ be such that $\Psi_{P,\mathrm{id}}(\sigma)\cdot D_i = D_i$ for all $i\in \{1, \ldots, g\}$. The relations given in \ref{re_Delta_defintion} show that 
$\Psi_{P,\mathrm{id}}(\sigma)\cdot D_{g+1} = D_{g+1}$ as well as $\Psi_{P,\mathrm{id}}(\sigma)\cdot D_{g+2} = D_{g+2} = 0$. Now, let $n \in \{1,\ldots, m\}$ be $n:=\sigma^{-1}(m)$. 
Then the condition $\Psi_{P,\mathrm{id}}(\sigma)\cdot D_n= D_n$ implies 
\[
Q_n - Q_m \sim Q_{m} - Q_{\sigma(n)} 
\]
and hence $D_n = - D_{\sigma(n)}$. This is not one of relations from \ref{re_Delta_defintion}, which are shown to be all the relations in Theorem~\ref{thm_delta_basis}, so we must have $D_n = 0$ and hence $n=m$. So now, for any $i\in \{1, \ldots, g\}$, the condition $\Psi_{P,\mathrm{id}}(\sigma)\cdot D_i = D_i$  implies 
\[
Q_i - Q_m \sim Q_{\sigma(i)} - Q_{m} 
\]
and $D_i = D_{\sigma(i)}$. Again, the relations from Definition~\ref{re_Delta_defintion} and 
Theorem~\ref{thm_delta_basis} show that this implies $\sigma(i) = i$ for any $i\in \{1, \ldots, g\}$. Hence we have that $\sigma = \mathrm{id}$. Thus, we have shown that $\Psi_{P,\mathrm{id}}$ has trivial kernel and so is injective. 
\end{proof}

\begin{lem}\label{lem_psi_isomorphism_onto_irr_component}
For $A \in \Sp(2g, \FF_3)$ the morphism $\psi_{A}$ is an isomorphism onto an irreducible component of $\Sup^{3}_{g,\mathsf{m}}[3]$. 
\end{lem}
\begin{proof}
Denote the quotient map by $\mathsf{q}_{\mathsf{m}}: M_{0,m} \rightarrow M_{0,m}/\mathsf{A}_{\mathsf{m}}$. We have that both $\mathsf{q}_{\mathsf{m}}$  and the morphism ${\mathsf{F}_{\mathsf{m}}}$ forgetting the level-3 structure are finite. Hence we have that $\psi_A$ is finite as well (see for example \cite[Tag 01WJ \& Tag 035D]{stacksproject}). Thus, since Lemma~\ref{lem_psi_A_injective} shows that $\psi_A$ is  injective  we have that it is a closed immersion (see for example \cite[Tag 03BB]{stacksproject}).

Since the quotient map $\mathsf{q}_{\mathsf{m}}$ and ${\mathsf{F}_{\mathsf{m}}}$ are both finite and surjective, we have that $M_{0,m}$ and  $\Sup^{3}_{g,\mathsf{m}}[3]$ are both the same dimension (see for example \cite[Tag 01WJ \& Tag 0ECG]{stacksproject}). Now, let $\mathcal{A}\subseteq \Sup^{3}_{g,\mathsf{m}}$ be the irreducible component that contains the image of $\psi_A$, and let $\varepsilon : M_{0,m}\rightarrow \mathcal{A}$ be the associated morphism. Since $M_{0,m}$ is irreducible, we must have that $\varepsilon$ is surjective.

On top of this, both $\mathsf{q}_{\mathsf{m}}$ and ${\mathsf{F}_{\mathsf{m}}}$ are étale, so by the vanishing of cotangent complexes we have that $\varepsilon$ is étale as well. In conclusion, $\varepsilon$ is a surjective flat closed immersion and hence $\varepsilon$ is an isomorphism (see \cite[Tag 04PW]{stacksproject}). 
\end{proof}

\begin{cor}\label{cor_forgetful_preimage_size}
Let $x=[\pi] \in \Sup^3_g$ be a geometric point and let $A\in \Sp(2g,\FF_3)$. If $[\pi, \eta]$ and $[\pi, \eta']$ are both in the image $\mathrm{Im}\, \psi_A$ then there is a unique $B\in \mathsf{A}_{\mathsf{m}}$ such that $\eta' = T_B\circ \eta$ (where $T_B$ is as defined in \ref{re_morphism_M_Sup}). Moreover, we have that 
\[
\big|\,\mathrm{Im}\, \psi_A \cap \mathsf{F}_{\mathsf{m}}^{-1}(x)\,\big| = | \mathsf{A}_{\mathsf{m}}|. 
\]
\end{cor}
\begin{proof}
This follows from Lemma~\ref{lem_psi_isomorphism_onto_irr_component} since $\psi_A$ is an isomorphism on an irreducible component of $\Sup^3_{g,\mathsf{m}}[3]$ and the morphism $M_{0,m}\rightarrow \Sup^3_{g,\mathsf{m}}$ described by diagram (\ref{sup_w_level_3_commuting_diagram}) is the fixed-point-free quotient of $M_{0,m}$ by $\mathsf{A}_{\mathsf{m}}$. 
\end{proof}

\begin{lem} \label{lem_psi_A_disjoint_images}
Let $A,A' \in \Sp(2g,\FF_3)$. If the images $\mathrm{Im}\,\psi_{A}$ and $\mathrm{Im}\,\psi_{A'}$ are not disjoint, then
$\mathrm{Im}\, \psi_{A} = \mathrm{Im}\, \psi_{A'} $. 
\end{lem}
\begin{proof}
Suppose that the images $\mathrm{Im}\,\psi_{A}$ and $\mathrm{Im}\,\psi_{A'}$ are not disjoint and take a point  $[\pi, \eta] \in \Sup^3_g[3]$ in the intersection. Then, by the construction of $\psi_A$ and $\psi_{A'}$, we must have that $\eta = T_{A} \circ \tau $ and $\eta = T_{A'} \circ \tau' $ for points $[\pi,\tau],[\pi,\tau']\in \mathrm{Im}\,\psi_{\mathrm{id}}$. 

Hence by Corollary~\ref{cor_forgetful_preimage_size} we have that there is a $B\in  \mathsf{A}_{\mathsf{m}}$ such that $\tau' = T_B\circ \tau$. Combining that with the equality $\eta = T_{A} \circ \tau =  T_{A'} \circ \tau' $ we see that $T_{A} = T_{A'}\circ T_{B}$ and hence $A = A'B$. 

Now, by the definition of $\psi_{A}$ we have that $\psi_{A}= \Theta_{A'} \circ \psi_B$ and $\psi_{A'} = \Theta_{A'} \circ \psi_{\mathrm{id}}$. The result now follows from the observation that $\mathrm{Im}\, \psi_{\mathrm{id}} = \mathrm{Im}\, \psi_B$. 
\end{proof}

\begin{lem} \label{lem_sup_m[3]_explicit_moduli}
The scheme $\Sup^{3}_{g,\mathsf{m}}[3]$ has 
$
{|\Sp(2g,\FF_3)|}/{  |\mathsf{A}_{\mathsf{m}}|}
$
 connected components. Each connected component is irreducible and is isomorphic to $M_{0,m}$. 
\end{lem}
\begin{proof}
We begin by claiming that the morphism 
\[
\bigsqcup_{A\in \Sp(2g, \FF_3)} \psi_A : \bigsqcup_{A\in \Sp(2g, \FF_3)} M_{0,m} \longrightarrow \Sup^3_g[3]
\]
is surjective. So see this, take any geometric point $[\pi:C\rightarrow \PP^1,\eta']\in \Sup^3_g[3]$ and consider the point $[\pi]\in \Sup^3_g$. We know from diagram (\ref{sup_w_level_3_commuting_diagram}) that the morphism $M_{0,m} \rightarrow \Sup^3_g$ is surjective, so let $P \in M_{0,m}$ be be any point in the preimage of $[\pi]$ under this morphism. We now have that $\psi_{\mathrm{id}}(P) = [\pi,\eta]$ for some isometry $\eta: \FF_3^{2g} \rightarrow \Jac(C)[3]$. Considering the composition $\eta' \circ \eta^{-1} \circ \eta  = \eta'$ we observe that $\eta' \circ \eta^{-1}$ corresponds to symplectic change of basis $\FF^{2g}_3\rightarrow \FF^{2g}_3$ and hence $\eta' \circ \eta^{-1} = T_A$ for some $A\in \Sp(2g,\FF_3)$. This proves the claimed surjectivity. 

Let $\mathcal{B} \subset \Sp(2g,\FF_3)$ have the following properties:
\begin{enumerate}
\item $\bigsqcup_{A\in \mathcal{B}} \psi_A$ is surjective.
\item If $\mathcal{B}' \subset \Sp(2g,\FF_3)$ has the property that $\bigsqcup_{A\in \mathcal{B}'} \psi_A $ is surjective then we have $|\mathcal{B}| \leq |\mathcal{B}'|$.
\end{enumerate}
Such a $\mathcal{B}$ will always exist but may not be unique. Moreover, combining property (ii) with 
Lemma~\ref{lem_psi_A_disjoint_images} we must have that for each pair $A,A'\in \mathcal{B}$ with $A\neq A'$ that the images $\mathrm{Im}\,\psi_A$ and $\mathrm{Im}\,\psi_{A'}$ are disjoint. 

Now, letting $x \in \Sup^3_g$ be a geometric point, we have that
\[
\mathsf{F}_{\mathsf{m}}^{-1}(x) = \bigsqcup_{A\in\mathcal{B}} \big(\,\mathrm{Im}\, \psi_A \cap \mathsf{F}_{\mathsf{m}}^{-1}(x)\,\big).
\]
Hence using Corollary~\ref{cor_forgetful_preimage_size}  and the fact that $|\mathsf{F}_{\mathsf{m}}^{-1}(x)| = |\Sp(2g,\FF_3)|$ we have 
\[
 |\Sp(2g,\FF_3)|= \sum_{A\in\mathcal{B}} \big|\,\mathrm{Im}\, \psi_A \cap \mathsf{F}_{\mathsf{m}}^{-1}(x)\,\big| 
 = |\mathcal{B}| \cdot |\mathsf{A}_{\mathsf{m}}|.
\]
Since $|\mathcal{B}|$ is the number of connected components, the desired result now follows. 
\end{proof}

Theorem~\ref{mainthm}
now follows immediately from Lemma \ref{lem_sup_m[3]_explicit_moduli} and the description of the indexing set $\mathsf{M}$ from \ref{re_p=3_specifics}. 
We also summarize what we have obtained in alternative way to highlight the analogy
with Theorem~\ref{Hyperelliptic_w_level_structure_decomposition_theorem}.

\begin{thm}[\textit{Connected components of $\Sup^{3}_{g}[3]$}] \label{Superelliptic_w_level_structure_decomposition_theorem} 
For  $\mathsf{m}\in\mathsf{M}$, denote the quotient set $\mathrm{Sp}(2g,\FF_p)/\mathsf{A}_{\mathsf{m}}$ by
$\mathscr{C}_{\mathsf{m}}$. Then, if $X_{c}$ denotes a copy of $M_{0,m}$ for each $c \in \mathscr{C}_{\mathsf{m}}$,  there is an isomorphism of schemes
\begin{equation*}
\Sup^{3}_{g}[3] \overset{\cong}{\longrightarrow} 
\bigsqcup_{\mathsf{m}\in\mathsf{M}} 
\bigsqcup_{c \in \mathscr{C}_{\mathsf{m}}} X_c.
\end{equation*}
In particular, $\Sup^{3}_{g}[3]$ is smooth.
\end{thm}

\subsection*{Acknowledgements}
The authors thank Jonas Bergstöm and Dan Petersen for useful conversations. 


\bibliographystyle{amsalpha}

\renewcommand{\bibname}{References}
\bibliography{references} 
\end{document}